\documentclass[final,3p,times]{elsarticle}

\usepackage{graphicx,amsmath,amsthm,amsfonts,latexsym,amssymb,hyperref,comment,subeqnarray}
\usepackage{hyperref,mathrsfs}
\usepackage[mathscr]{euscript}
\usepackage{colortbl}
\usepackage{todonotes}
\usepackage{algorithmic}
\let\chapter\section               
\usepackage{algorithm}  

\DeclareFontEncoding{FMS}{}{}
\DeclareFontSubstitution{FMS}{futm}{m}{n}
\DeclareFontEncoding{FMX}{}{}
\DeclareFontSubstitution{FMX}{futm}{m}{n}
\DeclareSymbolFont{fouriersymbols}{FMS}{futm}{m}{n}
\DeclareSymbolFont{fourierlargesymbols}{FMX}{futm}{m}{n}
\DeclareMathDelimiter{\tbar}{\mathord}{fouriersymbols}{152}{fourierlargesymbols}{147}

\newcommand{\jump}[1]{\left[\!\left[#1\right]\!\right]}

\def\NN{\hbox{\rlap{I}\kern.16em N}}
\def\NC{\hbox{\rlap{\kern.24em\raise.1ex\hbox
                  {\vrule height1.3ex width.9pt}}C}}

\def\ljump{{[\![}}
\def\rjump{{]\!]}}

\def\bn{{\bf n}}
\def\3bar{{|\hspace{-.01in}|\hspace{-.01in}|}}

\def\pT{{\partial T}}

\newtheorem{remark}{Remark}[section]
\newtheorem{theorem}{Theorem}[section]

\newtheorem{lemma}{Lemma}[section]
\numberwithin{equation}{section}

\begin{document}

\begin{frontmatter}



\title{An Immersed Weak Galerkin Method For Elliptic Interface Problems}


\author{Lin Mu\footnote{This author's research was supported in part by the U.S.~Department of Energy, Office of Science, Office of Advanced Scientific Computing Research, Applied Mathematics program under award number ERKJE45; and by the Laboratory Directed Research and Development program at the Oak Ridge National Laboratory, which is operated by UT-Battelle, LLC., for the U.S.~Department of Energy under Contract DE-AC05-00OR22725.}}
\address{Computer Science and Mathematics Division,
Oak Ridge National Laboratory, Oak Ridge, TN 37831, USA  (mul1@ornl.gov). 
}

\author{Xu Zhang\footnote{This author is partially supported by the National Science Foundation DMS-1720425.}}
\address{Department of Mathematics and Statistics, Mississippi State University, Mississippi State, MS 39762, USA (xuzhang@math.msstate.edu)}

\begin{abstract}
In this paper, we present an immersed weak Galerkin method for solving second-order elliptic interface problems. The proposed method does not require the meshes to be aligned with the interface. Consequently, uniform Cartesian meshes can be used for nontrivial interfacial geometry. We show the existence and uniqueness of the numerical algorithm, and prove the error estimates for the energy norm. Numerical results are reported to demonstrate the performance of the method. 
\end{abstract}

\begin{keyword}
immersed weak Galerkin \sep interface problems \sep Cartesian mesh \sep error estimate


\end{keyword}

\end{frontmatter}


\section{Introduction}
\label{S:1}

We consider the following elliptic interface equation
\begin{eqnarray}
-\nabla\cdot(\beta\nabla u)&=&f,~~~\mbox{in }~\Omega^{-}\cup\Omega^{+},\label{eq:pde}\\
u&=&g,~~\mbox{ on }~\partial\Omega,\label{eq:bc}
\end{eqnarray}
where the domain $\Omega\subset\mathbb{R}^2$ is separated by an interface curve $\Gamma$ into two subdomains $\Omega^{+}$ and $\Omega^{-}$. The diffusion coefficient $\beta({\bf x})$ is discontinuous across the interface. Without loss of generality, we assume $\beta({\bf x})$ is a piecewise constant function as follows
\begin{eqnarray*}
\beta({\bf x})=\begin{cases}
\beta^{-},\mbox{ if }{\bf x}\in\Omega^{-},\\
\beta^{+},\mbox{ if }{\bf x}\in\Omega^{+}.
\end{cases}
\end{eqnarray*}
The exact solution $u$ is required to satisfy the following homogeneous jump conditions
\begin{eqnarray}
\ljump u\rjump|_{\Gamma}&=&0,\label{eq:interface1}\\
\ljump{\beta}\nabla u\cdot{\bf n}\rjump|_{\Gamma}&=&0,\label{eq:interface2}
\end{eqnarray}
where ${\bf n}$ is the unit normal vector to the interface $\Gamma$. From now on, we define 
\begin{eqnarray*}
v=\begin{cases}
v^{-}({\bf x}),\mbox{ if }{\bf x}\in\Omega^{-},\\
v^{+}({\bf x}),\mbox{ if }{\bf x}\in\Omega^{+},
\end{cases}
\end{eqnarray*}
and denote $\ljump v\rjump|_{\Gamma}=v^{+}|_{\Gamma}-v^{-}|_{\Gamma}.$ 

Interface problems arise in many applications in science and engineering. The elliptic problem \eqref{eq:pde} - \eqref{eq:interface2} represents a typical interface model problem since it captures many fundamental physical phenomena. To solve interface problems, in general, there are two classes of numerical methods. The first class of methods uses interface-fitted meshes, i.e., the solution mesh is tailored to fit the interface. Methods of this type include classical finite element methods (FEM) \cite{1996BrambleKing,1998ChenZou}, discontinuous Galerkin methods \cite{2002ArnoldBrezziCockburnMarini,1999RiviereWheelerGirault}, and the virtual element methods \cite{2013VeigaBrezziCangianiManziniMariniRusso, 2016VeigaBrezziMariniRusso}. The second class of methods use unfitted meshes which are independent of the interface.  Structured uniform meshes such as Cartesian meshes are often utilized in these methods. The advantages of unfitted-mesh methods often emerge when the interface is geometrically complicated for which a high-quality body-fitting mesh is difficult to generate; or the simulation involves a dynamic moving interface, which requires repeated mesh generation.  In the past decades, many numerical methods based on unfitted meshes have been developed. For instance, the immersed interface methods \cite{1994LevequeLi, 2006LiIto}, cut finite element methods \cite{2015BurmanClausHansboLarsonMassing,2002HansboHansbo}, multi-scale finite element methods  \cite{2010ChuGrahamHou,1999HouWuCai}, extended finite element methods \cite{2001DolbowMoesBelytschko, 1999MoesDolbowBelytschko}, to name only a few.

The immersed finite element method (IFEM) is an class of unfitted mesh methods for interface problems.  The main idea of the immersed finite element method is to locally adjust the approximation function instead of solution mesh to resolve solution around the interface. The IFEM was first developed for elliptic interface problems \cite{2017AdjeridGuoLin, 2017CaoZhangZhang, 1998Li, 2003LiLinWu, 2001LinLinRogersRyan} and was recently applied to other interface model problems such as elasticity system \cite{2012LinZhang, 2013LinSheenZhang}, Stokes flow \cite{2015AdjeridChaabaneLin}, parabolic moving interface problems \cite{2013HeLinLinZhang, 2013LinLinZhang1}, etc. Recently, this \textit{immersed} idea has also been used in various numerical algorithms other than classical conforming FEM, such as nonconforming IFEM \cite{2010KwakWeeChang,2015LinSheenZhang}, immersed Petrov-Galerkin methods \cite{2005HouLiu,2010HouWangWang}, immersed discontinuous Galerkin methods \cite{2014HeLinLin,2015LinYangZhang1}, and immersed finite volume methods \cite{2017CaoZhangZhangZou,2009HeLinLin}. 

The weak Galerkin (WG) methods are a new class of finite element discretizations for solving partial differential equations (PDE) \cite{2013MuWangWangYe,2013WangYe}. In the framework of the WG method, classical differential operators are replaced by generalized differential operators as distributions. Unlike the classical FEM that impose continuity in the approximation space, the WG methods enforce the continuity weakly in the formulation using generalized discrete weak derivatives and parameter-free stabilizers. The WG methods are naturally extended from the standard FEM for functions with discontinuities, and thus are more advantageous over FEM in several aspects \cite{2018LiuTavenerWang,2015MuWangYe}. For instance, high-order WG spaces are usually constructed more conveniently than conforming FEM spaces since there is no continuity requirement on the approximation spaces. Also, the relaxation of the continuity requirement enables easy implementation of WG methods on polygonal meshes, and more flexibility for $h$- and $p$- adaptation. Moreover, the weak Galerkin methods is absolutely stable and there is no tuning parameter in the scheme, which is different from interior penalty discontinuous Galerkin (IPDG)methods. 


Recently, the WG methods have been studied for elliptic interface problems \cite{2013MuWangWeiYeZhao,2016MuWangYeZhao}. These WG methods require that solution mesh to be aligned with interface in order to get the optimal convergence. In this article, we will develop an immersed weak Galerkin (IWG) methods for elliptic interface problems. The proposed IWG method combines the advantages from both immersed finite element approximation and the weak Galerkin formulation. One apparent advantage of our IWG method over standard WG method is that it can be applied on unfitted meshes such as Cartesian meshes for solving elliptic interface problems. Comparing with the immersed IPDG methods \cite{2014HeLinLin,2015LinYangZhang1}, the matrix assembling in the IWG method assembles is more efficiently because all computation can be done locally within an element without exchange information from neighboring elements. 

The rest of the article is organized as follows. In Section 2, we recall the $P_1$ immersed finite element spaces that will be used to construct the WG approximation spaces. In Section 3, we introduce the IWG algorithm and discuss the well-posedness of the discretized problem. Section 4 is dedicated to the error analysis of the IWG algorithm. We will show that the errors measured in energy norm obey the optimal rate of convergence with respect to the polynomial degree of approximation space. In Section 5, we provide several numerical examples to demonstrate features of our IWG method.

\section{Immersed Finite Element Functions and Weak Galerkin Methods}

In this section, we introduce notations to be used in this article, and review the basic ideas of weak Galerkin methods, and immersed finite element spaces. Throughout this paper, we adopt notations of standard Sobolev spaces. 
For $m>1$, and any subset $G\subset\Omega$ that is cut through by the interface $\Gamma$, we define the following Hilbert spaces
\begin{eqnarray*}
\tilde{H}^m(G)=\{u\in H^1(G): u|_{G\cap\Omega^s}\in H^m(G\cap\Omega^s), \ s=+\mbox{ or }-\}
\end{eqnarray*}
equipped the norm and semi-norm:
\begin{eqnarray*}
\|u\|_{\tilde H^{m}(G)} =  \|u\|_{m,G\cap\Omega^+} + \|u\|_{m,G\cap\Omega^-},~~~~
|u|_{\tilde H^{m}(G)} =  |u|_{m,G\cap\Omega^+} + |u|_{m,G\cap\Omega^-}.
\end{eqnarray*}
\subsection{Immersed Finite Element Spaces}
Let $\mathcal{T}_h$ be a shape-regular triangular mesh of the domain $\Omega$. For every element $T\in\mathcal{T}_h$, we denote by $h_T$ its diameter. The mesh size of $\mathcal{T}_h$ is defined by $h=\max_{T\in\mathcal{T}_h} h_T$. Since the mesh $\mathcal{T}_h$ is independent of the interface, we often use Cartesian triangular mesh for simplicity, see Figure \ref{fig:mesh}. The interface $\Gamma$ may intersect with some elements in $\mathcal{T}_h$, which are called interface elements. The rest of elements are called regular elements, see Figure \ref{fig:ReAndIre}. The collections of interface elements and regular elements, are denoted by $\mathcal{T}_h^I$ and $\mathcal{T}_h^R$, respectively. Denote by $\mathcal{E}_h$ the set of all edges in $\mathcal{T}_h$, and let $\mathcal{E}_{h}^0=\mathcal{E}_h\backslash\partial\Omega$ be the set of all interior edges. 

\begin{figure}[!tb]
\begin{center}
\includegraphics[width=0.3\textwidth]{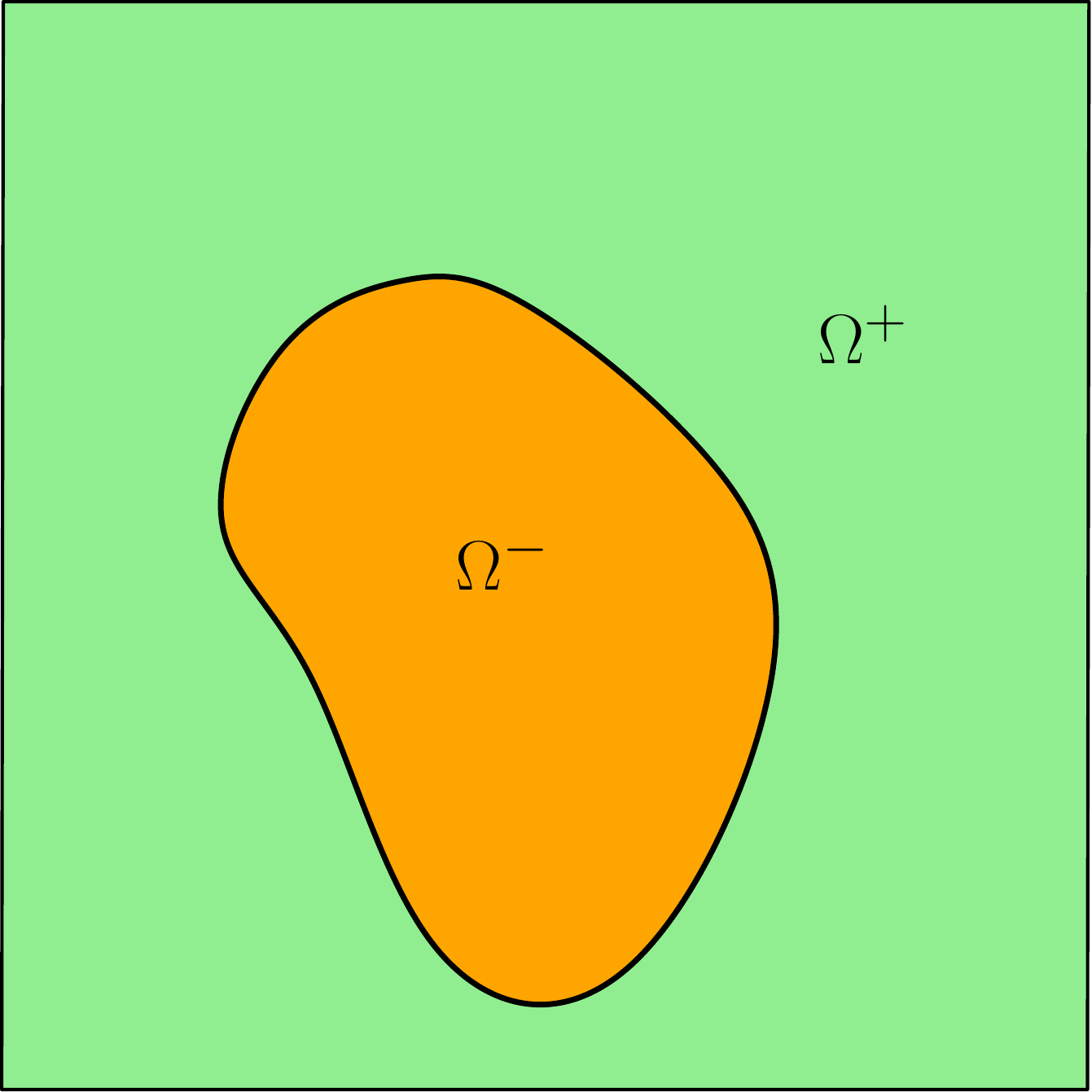}\quad
\includegraphics[width=0.3\textwidth]{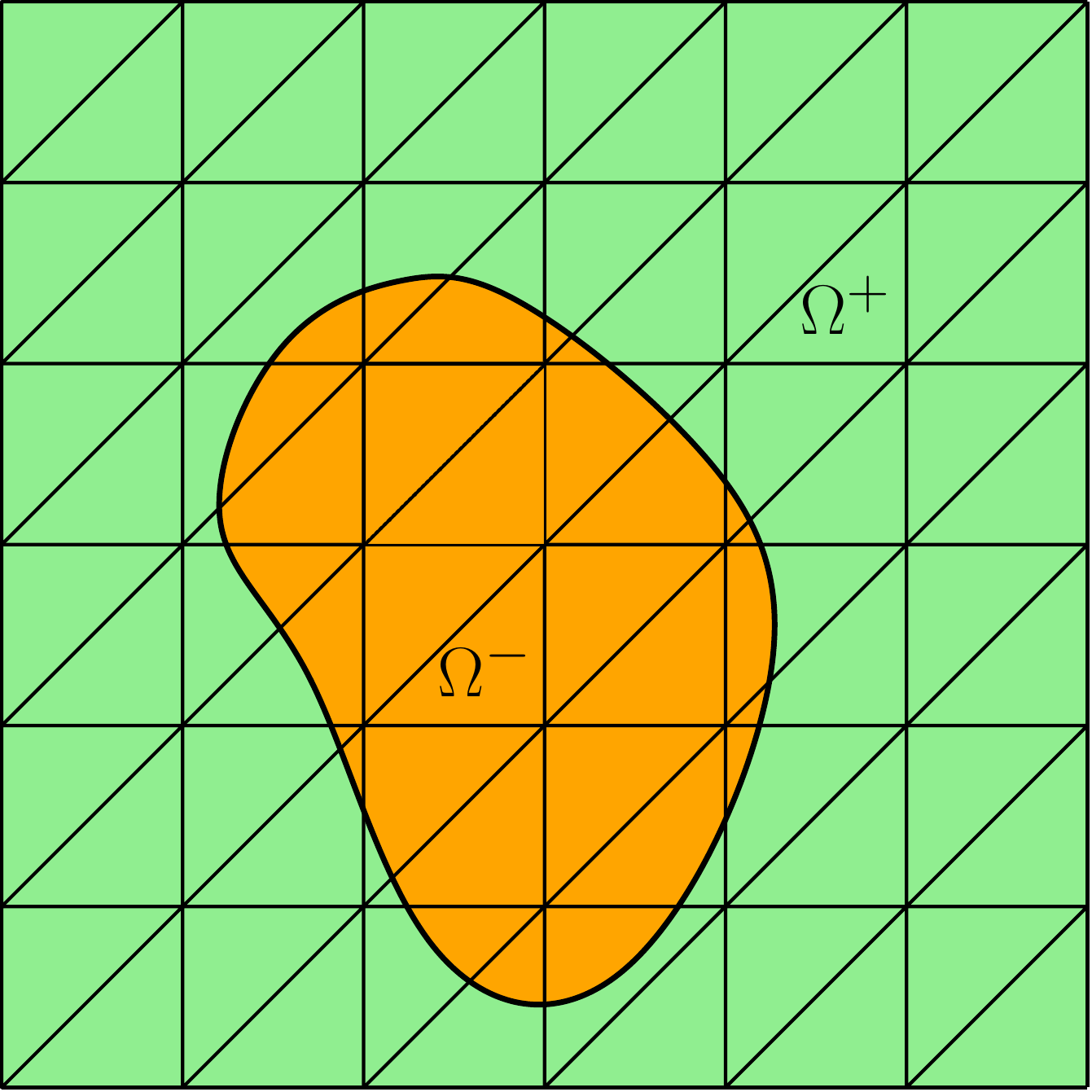}
\end{center}
\caption{Plots of interface $\Gamma$ and a Cartesian triangular mesh.}\label{fig:mesh}
\end{figure}

\begin{figure}[!tb]
\begin{center}
\includegraphics[width=0.3\textwidth]{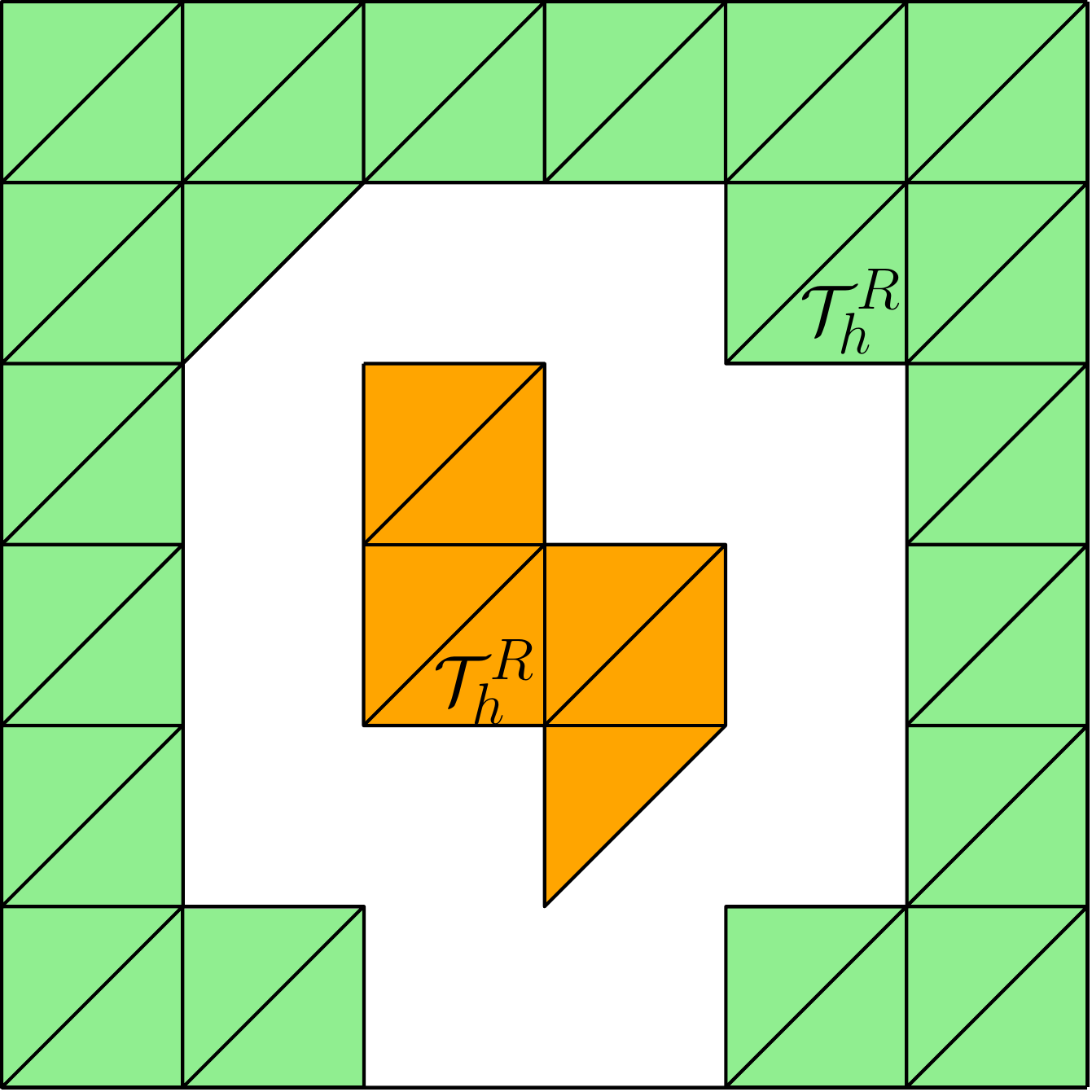}\quad
\includegraphics[width=0.3\textwidth]{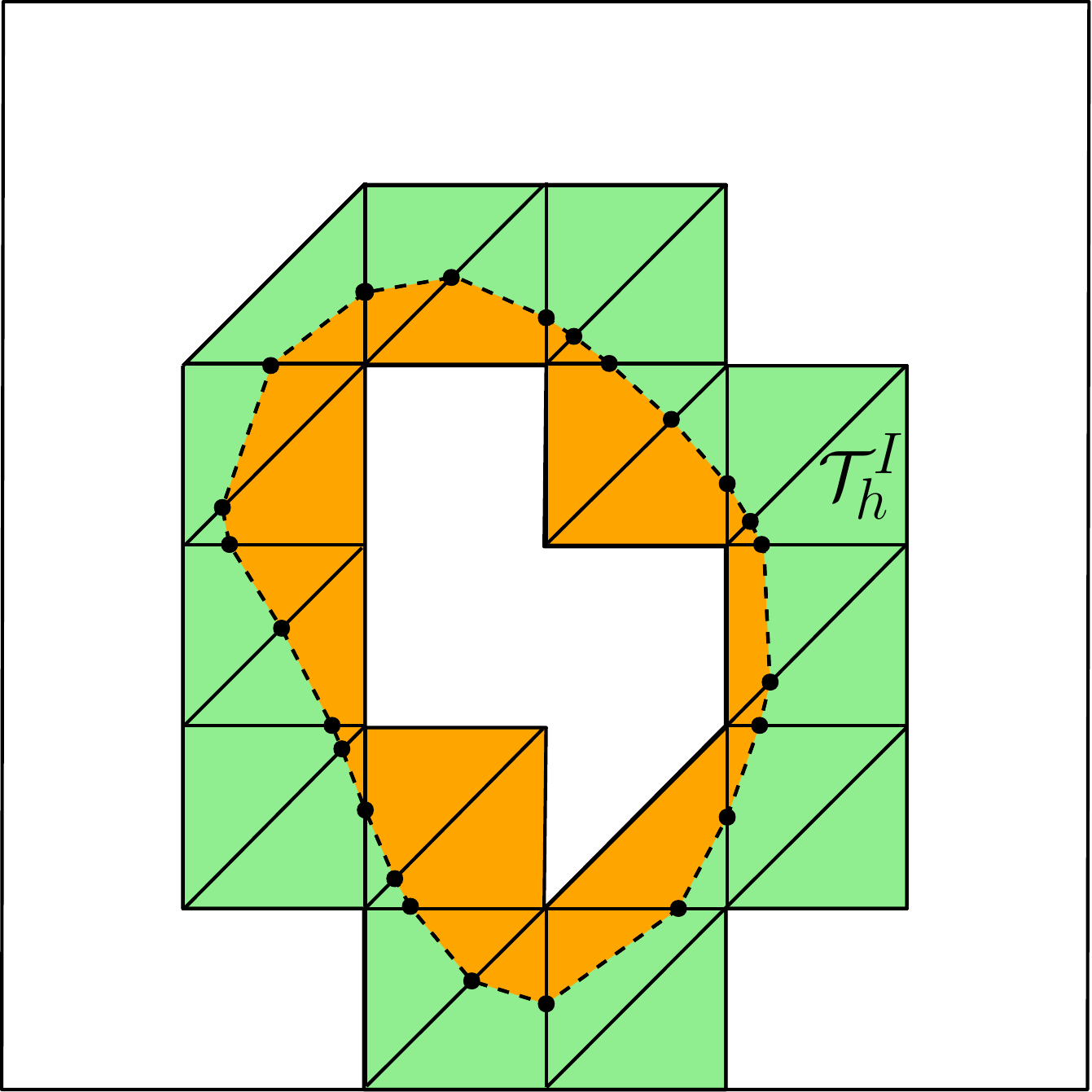}
\end{center}
\caption{Plots of regular elements $\mathcal{T}_h^R$ and interface elements $\mathcal{T}_h^I$.}\label{fig:ReAndIre}
\end{figure}

Without of generality, we assume that $\mathcal{T}_h$ satisfies the following hypotheses, when the mesh size $h$ is small enough:
\begin{itemize}
\item[(H1).] The interface $\Gamma$ cannot intersect an edge of any element at more than two points unless the edge is part of $\Gamma.$
\item[(H2).]If $\Gamma$ intersects the boundary of an element at two points, these intersection points must be on different edges of this element.
\item[(H3).]The interface $\Gamma$ is a piecewise $C^2$- function, and the mesh $\mathcal{T}_h$ is formed such that the subset of $\Gamma$ in every interface element $T\in \mathcal{T}_h^I$ is $C^2$-continuous.
\item[(H4).]When the mesh size $h$ is small enough, the number of interface elements is of order $O(h^{-1})$.
\end{itemize}


To be self-contained, we briefly recall the linear IFE space introduced in \cite{2004LiLinLinRogers, 2003LiLinWu}. Let $T\in\mathcal{T}_h^I$ be an interface element. Denote the three vertices of $T$ by $A_1$, $A_2$, and $A_3$. The interface curve $\Gamma$ cut the element $T$ at two intersection points $D$, $E$. The line segment $\overline{DE}$ divide the element $T$ into two sub-elements $T^-$ and $T^+$. See Figure \ref{fig: interface element} for a typical interface triangle.

\begin{figure}[htb]
\begin{center}
\begin{tikzpicture}[scale=0.8]
\draw (0,0) node[below left] {${A}_1$} --
(5,0.5) node[below right] {${A}_2$} --
(2,5) node[above right] {${A}_3$} -- cycle;
\draw (1,2.5)node [left] {${D}$}--(4,2) node[right] {${E}$};
\draw[red] (1,2.5) to [out=-55,in=-120](2.2,2.3) to [out=60,in=150] (4,2); 
\node at (2,1) {$T^-$};
\node at (2.3,3.4) {$T^+$};
\begin{scope}[>=latex]
\draw[->] (3,1.3) -- (2,2.1) node[at start,right] {$\Gamma$};
\end{scope}
\end{tikzpicture}
\caption{A typical triangular interface element.} \label{fig: interface element}
\end{center}
\end{figure}
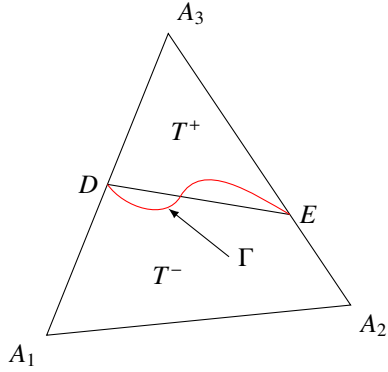

The linear IFE functions are constructed by incorporating the interface jump conditions.
Specifically, three linear IFE shape functions $\phi_i, i = 1, 2, 3$ associated with the vertices of $A_i, i = 1, 2, 3$ are constructed in the form of
\begin{equation}\label{eq: linear IFE function}
    \phi_{i}(x,y) =
    \left\{
      \begin{array}{cc}
        \phi_{i}^{+}(x,y) = a_i^+ + b_i^+x + c_i^+y,~~~~  &\text{if}~ (x,y)\in T^+, \vspace{1mm}\\
        \phi_{i}^{-}(x,y) = a_i^- + b_i^-x + c_i^-y,~~~~  &\text{if}~ (x,y)\in T^-, \vspace{1mm}\\
      \end{array}
    \right.
\end{equation}
satisfying the following conditions:
\begin{itemize}
  \item nodal value condition
  \begin{equation}\label{eq: nodal condition}
    \phi_{i}(A_j) = \delta_{ij},~~~~~i,j = 1,2,3.
\end{equation}
  \item continuity of the function
  \begin{equation}\label{eq: continuity1}
    \jump{\phi_{i}{(D)}} =0, ~~~~\jump{\phi_{i}{(E)}} = 0.
\end{equation}
  \item continuity of normal component of flux
  \begin{equation}\label{eq: continuity2}
    \jump{\beta\frac{\partial \phi_{i}}{\partial n}} = 0.
  \end{equation}
\end{itemize}
It has been shown \cite{2004LiLinLinRogers} that conditions specified in \eqref{eq: nodal condition} - \eqref{eq: continuity2} can uniquely determine these shape functions in \eqref{eq: linear IFE function}. Then, on each interface element $T \in \mathcal{T}_h^I$, we define the local IFE space
\begin{eqnarray}\label{eq: P1 IFE space}
\tilde P_1(T) = span\{\phi_1,\phi_2,\phi_3\}.
\end{eqnarray}

\subsection{Weak Functions}

The weak Galerkin method takes finite element functions in the form of two components, one in the interior and the other on the boundary. This means for a weak function $v$ defined on an element $T$,  
\begin{eqnarray*}
v=\begin{cases}
v_0,&\mbox{ in }T,\\
v_b,&\mbox{ on }\partial T.
\end{cases}
\end{eqnarray*}
For simplicity, we shall write $v$ as $v=\{v_0,v_b\}$ in short. 

We consider the following weak Galerkin finite element space
\begin{eqnarray*}
V_h:=\Big\{v=\{v_0,v_b\}:~v_0|_T\in P_1(T),\mbox{ if }T\in\mathcal{T}_h^R,v_0|_T\in \tilde{P}_1(T),\mbox{ if }T\in\mathcal{T}_h^I; ~v_b|_e\in P_0(e), e\subset\mathcal{E}_h\Big\}.
\end{eqnarray*}
Here ${P}_1(T)$ is the standard linear polynomial space, and $\tilde{P}_1(T)$ is the linear immersed finite element space on $T$ defined in \eqref{eq: P1 IFE space}. The $P_0(e)$ is the standard piecewise constant function on the edge $e$. 
Let $V_h^0$ be the subspace of $V_h$ consisting of finite element functions with vanishing boundary value:
\begin{eqnarray*}
V_h^0=\{v\in V_h: v_b=0 \mbox{ on }\partial\Omega\}.
\end{eqnarray*}


On each element $T\in\mathcal{T}_h$, define the projection operator $Q_h$ by
\begin{eqnarray*}
Q_hu=\{Q_0u,Q_bu\}\in V_h,
\end{eqnarray*}
where $Q_0$ is the Lagrange interpolation $C(T)$ to $P_1(T)$ or $\tilde P_1(T)$, depending on whether $T$ is a regular element or an interface element, and $Q_b$ is the $L^2$ projection from $L^2(e)$ to $P_0(e)$ for every edge $e$.

The immersed weak Galerkin method for the problem (\ref{eq:pde})-(\ref{eq:interface2}) is to seek: $u_h=\{u_{h0},u_{hb}\}\in V_h$ such that
\begin{eqnarray}\label{eq:scheme}
A(u_h,v)=(f,v_0),\ \forall v\in V_h^0,
\end{eqnarray}
where the bilinear form $A(u,v)$ is defined as
\begin{eqnarray}
A(u,v)&=&\sum_{T\in\mathcal{T}_h}\bigg( (\beta\nabla u_0,\nabla v_0)_T-\langle Q_b(\beta\nabla u_0\cdot\bn),v_0-v_b\rangle_\pT\notag\\
&&-\langle Q_b(\beta\nabla v_0\cdot\bn),u_0-u_b\rangle_\pT+\rho h^{-1}\langle Q_bu_0-u_b,Q_b v_0-v_b\rangle_\pT\bigg), \label{eq:blinear}
\end{eqnarray}
where $\rho$ is a positive constant.
\begin{remark}
On every regular element $T\in\mathcal{T}_h^R$ and $e\subset\pT$, we have $Q_b(\beta\nabla\phi_0\cdot\bn)=\beta\nabla\phi_0\cdot\bn$
simply because $\beta\nabla\phi_0\cdot\mathbf{n}$ is a constant .  
\end{remark}

\section{Well-posedness of Numerical Algorithm}
In this section, we present the existence and uniqueness of the proposed immersed weak Galerkin method.

\begin{lemma}
The following inequality holds on every element $T\in\mathcal{T}_h$
\begin{equation}\label{eq: insertQbv0}
\|v_0-v_b\|_\pT^2\le h\|\nabla v_0\|_T^2+\|Q_bv_0-v_b\|_\pT^2,~~~\forall v\in V_h.
\end{equation}
\end{lemma}

\begin{proof}
We note that the inequality \eqref{eq: insertQbv0} is a standard estimate for $T\in \mathcal{T}^R_h$.  On an interface element $T\in \mathcal{T}^I_h$, we note that $v_0\in H^1(T)$. Therefore, applying the triangular inequality and trace inequality yields
\begin{equation*}
\|v_0-v_b\|_{\pT}^2\le\|v_0-Q_bv_0\|_{\pT}^2+\|Q_bv_0-v_b\|_{\pT}^2
\le h\|\nabla v_0\|_T^2+\|Q_bv_0-v_b\|_{\pT}^2.
\end{equation*}
\end{proof}

\begin{lemma}
For all $v\in V_h$, $T\in\mathcal{T}_h$, and $e\subset\pT$, the following inequality holds,
\begin{eqnarray}
\|Q_b v\|_e\le \|v\|_e~~~\forall v\in V_h.\label{eq:Qbv}
\end{eqnarray}
\end{lemma}
\begin{proof}
By the definition of $Q_b$ and Cauchy-Schwartz inequality, we obtain
\begin{eqnarray*}
\|Q_b v\|_e^2=\langle Q_b v,Q_b v\rangle_e=\langle v,Q_b v\rangle_e\le \|v\|_e\|Q_bv\|_e.
\end{eqnarray*}
\end{proof}

\begin{theorem}
The immersed weak Galerkin method (\ref{eq:scheme}) has a unique solution provided that $\rho$ is big enough.
\end{theorem}
\begin{proof} We show this well-posedness result by proving the continuity and coercivity of the bilinear form. For the continuity, we have  \begin{eqnarray*}
A(w,v)&=&\sum_{T\in\mathcal{T}_h^R\cup\mathcal{T}_h^I}\bigg( (\beta\nabla w_0,\nabla v_0)_T-\langle Q_b(\beta\nabla w_0\cdot\bn),v_0-v_b\rangle_\pT\notag\\
&&-\langle Q_b(\beta\nabla v_0\cdot\bn),w_0-w_b\rangle_\pT+h^{-1}\rho\langle Q_bw_0-w_b,Q_b v_0-v_b\rangle_\pT\bigg)\\
&=&\sum_{T\in\mathcal{T}_h^R\cup\mathcal{T}_h^I}\bigg( (\beta\nabla w_0,\nabla v_0)_T-\langle Q_b(\beta\nabla w_0\cdot\bn),Q_bv_0-v_b\rangle_\pT\notag\\
&&-\langle Q_b(\beta\nabla v_0\cdot\bn),Q_bw_0-w_b\rangle_\pT+h^{-1}\rho\langle Q_bw_0-w_b,Q_b v_0-v_b\rangle_\pT\bigg)\\
&\le&\sum_T\bigg(\|\beta^{1/2}\nabla w_0\|_T\|\beta^{1/2}\nabla v_0\|_T+(h\|\beta^{1/2}\nabla w_0\cdot\bn\|_\pT^2)^{1/2}(\beta h^{-1}\|Q_bv_0-v_b\|_e^2)^{1/2}\\
&&\quad+(h\|\beta^{1/2}\nabla v_0\cdot\bn\|_\pT^2)^{1/2}(\beta h^{-1}\|Q_bw_0-w_b\|_e^2)^{1/2}
+(h^{-1}\rho\|Q_bv_0-v_b\|_\pT^2)^{1/2}(h^{-1}\rho\|Q_bw_0-w_b\|_\pT^2)^{1/2}\bigg)\\
&\le&\sum_T \bigg(\|\beta^{1/2}\nabla w_0\|_T\|\beta^{1/2}\nabla v_0\|_T+\|\beta^{1/2}\nabla w_0\|_T(\beta h^{-1}\|Q_bv_0-v_b\|_e^2)^{1/2}
\\
&&+ \|\beta^{1/2}\nabla v_0\|_T(\beta h^{-1}\|Q_bw_0-w_b\|_e^2)^{1/2}+(h^{-1}\rho\|Q_bv_0-v_b\|_\pT^2)^{1/2}(h^{-1}\rho\|Q_bw_0-w_b\|_\pT^2)^{1/2}\bigg)\\
&\le&C\3bar w\3bar~ \3bar v\3bar.
\end{eqnarray*}

Then, we show the coercivity of the bilinear form. Note that
\begin{equation}\label{eq:Avv}
A(v,v)=\sum_{T\in\mathcal{T}_h} \bigg(\| \beta^{1/2}\nabla v_0\|_T^2-2\langle Q_b(\beta\nabla v_0\cdot\bn),v_0-v_b \rangle_{\pT}+\rho h^{-1}\|Q_bv_0-v_b\|_{\pT}^2\bigg).
\end{equation}

We have for $T\in\mathcal{T}_h^R\cup\mathcal{T}_h^I$
\begin{eqnarray*}
2\left\langle Q_b(\beta\nabla v_0\cdot\bn),v_0-v_b \right\rangle_{\pT}
&=&2\left\langle Q_b(\beta\nabla v_0\cdot\bn),Q_bv_0-v_b \right\rangle_{\pT}\\
&\le& 2\left(h\|Q_b(\beta^{1/2}\nabla v_0\cdot\bn)\|_{\pT}^2\right)^{1/2}\left(\beta h^{-1}\|Q_bv_0-v_b\|_{\pT}^2\right)^{1/2}\\
&\le& 2\left(\frac{h\|\beta^{1/2}\nabla v_0\cdot\bn\|_{\pT}^2}{2\epsilon}\right)+2\left(\frac{\beta\epsilon h^{-1}\|Q_bv_0-v_b\|_{\pT}^2}{2}\right)\\
&=& \frac{h\|\beta^{1/2}\nabla v_0\cdot\bn\|_{\pT}^2}{\epsilon}+\epsilon \beta h^{-1}\|Q_bv_0-v_b\|_{\pT}^2\\
&\le& (1/\epsilon)\|\beta^{1/2}\nabla v_0\|_T^2+\epsilon\beta h^{-1}\|Q_bv_0-v_b\|_{\pT}^2.
\end{eqnarray*}

Substituting the above inequality into (\ref{eq:Avv}), we obtain
\begin{eqnarray*}
A(v,v)\ge \sum_{T\in\mathcal{T}_h}(1-1/\epsilon)\|\beta^{1/2}\nabla v_0\|_T^2+(\rho-\epsilon\beta_{\max})h^{-1}\|Q_bv_0-v_b\|_{\pT}^2.
\end{eqnarray*}

Choosing $\epsilon=2$ and $\rho>2\beta_{\max}$ completes the proof of the coercivity. 
\end{proof}

\section{Error Analysis}
In this section, we derive the a priori error estimate for the immersed weak Galerkin method \eqref{eq:scheme}. We define the energy norm by
\begin{equation*}
\3bar v\3bar^2=\sum_{T\in\mathcal{T}_h^R\cup\mathcal{T}_h^I}\bigg(\|\beta^{1/2}\nabla v_0\|_T^2+ \rho h^{-1}\|Q_bv_0-v_b \|_{\partial T}^2\bigg).
\end{equation*}

First, we recall some trace inequalities on regular elements and interface elements. Let $T\in\mathcal{T}_h^R$ be a regular element and $e$ be an edge of $T$. The standard trace inequality holds for every function $v\in H^1(T)$: 
\begin{eqnarray}
\|v\|_e^2\le C\left(h_T^{-1}\|v\|_{0,T}^2+h_T\|\nabla v\|_{0,T}^2\right).
\end{eqnarray}
If $T\in\mathcal{T}_h^I$ is an interface element, the following lemma provides the trace inequalities of IFE functions \cite{2015LinLinZhang}.
\begin{lemma}
There exists a constant $C$ independent of the interface location such that for every linear IFE function $v\in \tilde{P}_1(T)$ the following inequalities hold:
\begin{eqnarray}
\|\beta v_p\|_{0,e}&\le& Ch^{1/2}|T|^{-1/2}\|\sqrt{\beta}\nabla v\|_{0,T},\ p=x,y\\
\|\beta\nabla v\cdot\bn_e\|_{0,e}&\le& Ch^{1/2}|T|^{-1/2}\|\sqrt{\beta}\nabla v\|_{0,T}.\label{eq:IFE_flux}
\end{eqnarray}
\end{lemma}

The next two lemmas provide the interpolation error estimates for linear IFE spaces  \cite{2004LiLinLinRogers, 2015LinLinZhang}.

\begin{lemma}
Let $T\in\mathcal{T}_h^I$ be an interface element. There exists a constant $C$, independent of interface location, such that the interpolation $I_hu$ in the IFE space $\tilde P_1(T)$ has the following error bound:
\begin{eqnarray}\label{eq: int error elem}
\|u-I_hu\|_{0,T}+h\|u-I_hu\|_{1,T}\le Ch^2\|u\|_{\tilde{H}^2(T)},~~~\forall u\in \tilde{H}^2(T).
\end{eqnarray}
\end{lemma}


\begin{lemma}
For every $u\in\tilde{H}^3(\Omega)$ satisfying the interface jump conditions, there exists a constant $C$ independent of the interface such that its interpolation $I_hu$ in the IFE space $V_h$ has the following bound:
\begin{eqnarray}
\|\beta\nabla (u-I_h u)|_T\cdot\bn_e\|_e^2\le C\Big(h^2\|u\|_{\tilde{H}^3(\Omega)}^2+h\|u\|_{\tilde{H}^2(T)}^2\Big),
\end{eqnarray}
where $T$ is an interface element and $e$ is one of its interface edge.
\end{lemma}

Next, we present some lemmas that will be used in our error analysis.
%

\begin{lemma}
There exists a constant $C$ such that
\begin{eqnarray}\label{eq:stab-estimate}
|S(Q_hw,v)|\le C h\|w\|_{\tilde{H}^2(\Omega)}\3bar v\3bar,~~~
\forall w\in \tilde H^{2}(\Omega),~~\forall v\in V_h
\end{eqnarray}
where $S(Q_hw,v)=\sum_Th^{-1}\langle Q_bQ_0w-Q_bw,Q_bv_0-v_b\rangle_\pT$.
\end{lemma}

\begin{proof}

Using the Cauchy-Schwarz inequality, trace inequality, and the interpolation error bound \eqref{eq: int error elem}, 
we have
\begin{eqnarray*}
\left|S(Q_hw,v)\right|
&=&\left|\sum_{T\in\mathcal{T}_h^R\cup\mathcal{T}_h^I}h^{-1}\langle Q_b(Q_0w)-Q_bw,Q_bv_0-v_b\rangle_{\partial T}\right|
=\left|\sum_{T\in\mathcal{T}_h^R}h^{-1}\langle Q_0w-w,Q_bv_0-v_b\rangle_{\partial T}\right|\\
&\le& C\bigg(\sum_{T\in\mathcal{T}_h^R}h^{-2}\|Q_0w-w\|_T^2+\|\nabla(Q_0w-w)\|_T^2\bigg)^{1/2}\bigg(\sum_{T\in\mathcal{T}_h^R}h^{-1}\|Q_bv_0-v_b\|_{\partial T}^2\bigg)^{1/2}\\
&\le& Ch\|w\|_{\tilde H^2(\Omega)}\3bar v\3bar.
\end{eqnarray*}
\end{proof}

\begin{lemma}
There exists a constant $C$ such that
\begin{eqnarray} \label{eq: proj sum}
\sum_{T\in\mathcal{T}_h}\|Q_0 u-u\|_\pT^2\le Ch^3\|u\|_{\tilde{H}^2(\Omega)}^2
\end{eqnarray}
\end{lemma}

\begin{proof}
Applying the trace inequality and the interpolation error bound \eqref{eq: int error elem},
we have 
\begin{equation*}
\|Q_0 u-u\|_\pT \leq C\left(h^{1/2}|Q_0u-u|_{1,T}+h^{-1/2}\|Q_0u-u\|_{0,T}\right)
\le Ch^{3/2}\|u\|_{\tilde H^2(T)}. 
\end{equation*}
Squaring both sides and summing over all elements lead to  the estimate \eqref{eq: proj sum} 
\end{proof}

\begin{lemma}
Let $u_h=\{u_0,u_b\}$ and $u$ be the solutions to problem (\ref{eq:scheme}) and (\ref{eq:pde})-(\ref{eq:bc}), respectively. Let $Q_hu=\{Q_0u,Q_bu\}$ be the projection of $u$ to the finite element space $V_h$. Then, for every function $v\in V_h^0$, one has the following error equation
{\begin{eqnarray}\label{eq:error-eq}
A(Q_hu-u_h,v)=L_u(v)+S(Q_hu,v),
\end{eqnarray} }
where
\begin{eqnarray}\label{eq:luv}
L_u(v)=\sum_{T\in\mathcal{T}_h}\bigg(\langle u-Q_0u,\beta\nabla v_0\cdot\bn-Q_b(\beta\nabla v_0\cdot\bn)\rangle_{\pT}+\langle\beta\nabla u\cdot\bn-Q_b(\beta\nabla Q_0u\cdot\bn),v_0-v_b\rangle_{\pT}\bigg).
\end{eqnarray}
\end{lemma}

\begin{proof} For any $v = (v_0,v_b)\in V_h^0$, we multiply (\ref{eq:pde}) by $v_0$ to obtain
\begin{eqnarray*}
(f,v_0)&=&\sum_{T\in T_h}(-\nabla \cdot\beta\nabla u,v_0)_T=\sum_{T\in T_h}\bigg(-\langle\beta \nabla u\cdot\bn,v_0\rangle_{\pT}+(\beta \nabla u,\nabla v_0)_T\bigg)\\
&=&\sum_{T\in T_h}\bigg(-\langle\beta\nabla u\cdot\bn,v_0\rangle_{\pT}+\langle u,\beta \nabla v_0\cdot\bn\rangle_{\pT}-(u,\nabla\cdot\beta\nabla v_0)_T\bigg)\\
&=&\sum_{T\in T_h}\bigg(-\langle\beta\nabla u\cdot\bn,v_0\rangle_{\pT}+\langle u,\beta\nabla v_0\cdot\bn\rangle_{\pT}-(Q_0 u,\nabla\cdot\beta \nabla v_0)_T\bigg)\\
&=&\sum_{T\in T_h}\bigg(-\langle\beta\nabla u\cdot\bn,v_0\rangle_{\pT}+\langle u,\beta\nabla v_0\cdot\bn\rangle_{\pT}-\langle Q_0 u,\beta\nabla v_0\cdot\bn\rangle_{\pT}+(\nabla Q_0 u,\beta\nabla v_0)_T\bigg)\\
&=&\sum_{T\in T_h}\bigg((\beta\nabla Q_0 u,\nabla v_0)_T-\langle Q_0u,\beta\nabla v_0\cdot\bn \rangle_{\pT}+\langle u,\beta\nabla v_0\cdot\bn\rangle_{\pT}-\langle\beta\nabla u\cdot\bn,v_0-v_b\rangle_{\pT}\bigg).
\end{eqnarray*}
The last equation is because $v_b$ is a constant on every edge, and the flux $\beta\nabla u\cdot \mathbf{n}$ is continuous. 
Then by the definition of the bilinear form \eqref{eq:blinear}, 
we have
\begin{eqnarray}
A(Q_h u,v)&=&(f,v_0)+ \sum_{T\in\mathcal{T}_h}\bigg(\langle u-Q_0u,Q_b(\beta\nabla v_0\cdot\bn)-\beta\nabla v_0\cdot\bn\rangle_{\pT}+\langle\beta\nabla u\cdot\bn-Q_b(\beta\nabla Q_0u\cdot\bn),v_0-v_b\rangle_{\pT}\notag\\
&&+\rho h^{-1}\langle Q_bQ_0u-Q_bu,Q_bv_0-v_b\rangle_{\pT}\bigg).\label{eq:L2u}
\end{eqnarray}
Subtracting (\ref{eq:scheme}) from the above equation, it is obtained that
\begin{eqnarray*}
A(Q_hu-u_h,v)
&=&\sum_{T\in\mathcal{T}_h}\bigg(\langle u-Q_0u,Q_b(\beta\nabla v_0\cdot\bn)-\beta\nabla v_0\cdot\bn\rangle_{\pT}
+\langle\beta\nabla u\cdot\bn-Q_b(\beta\nabla Q_0u\cdot\bn),v_0-v_b\rangle_{\pT}\\
&&+\rho h^{-1}\langle Q_bQ_0u-Q_bu,Q_bv_0-v_b\rangle_{\pT}\bigg)\\
&=& L_u(v)+S(Q_hu,v),
\end{eqnarray*}
which completes the proof.
\end{proof}

\begin{lemma}
The linear form $L_u(v)$ in \eqref{eq:luv} has the following error estimate
\begin{eqnarray}\label{eq: error luv}
L_u(v) \leq Ch\|u\|_{\tilde{H}^3(\Omega)} \3bar v\3bar.
\end{eqnarray}
where the constant $C$ is independent of the interface location.
\end{lemma}
\begin{proof}
In \eqref{eq:luv}, we denote $L_u(v)=I+II$. By Cauchy-Schwartz inequality, (\ref{eq:Qbv}), and \eqref{eq: proj sum}, we obtain
\begin{eqnarray*}
I&=&\sum_T \langle u-Q_0u,Q_b(\beta\nabla v_0\cdot\bn)-\beta\nabla v_0\cdot\bn\rangle_{\pT}\\
&\le& \sum_T(h^{1/2}\|Q_b(\beta\nabla v_0\cdot\bn)\|_\pT+h^{1/2}\|\beta\nabla v_0\cdot\bn\|_\pT) (h^{-1/2}\|Q_0u-u\|_\pT)\\
&\le& Ch\|u\|_{\tilde{H}^2(\Omega)} \sum_T\| \beta^{1/2}\nabla v_0\|_T\\
&\le& Ch\|u\|_{\tilde{H}^2(\Omega)} \3bar v\3bar.
\end{eqnarray*}

Next, by the trace inequality
\begin{eqnarray*}
II&=&\sum_T\langle\beta\nabla u\cdot\bn-Q_b(\beta\nabla Q_0u\cdot\bn),v_0-v_b\rangle_{\pT}\\
&=&\sum_T\langle\beta\nabla u\cdot\bn-Q_b(\beta\nabla Q_0u\cdot\bn),v_0-Q_bv_0\rangle_\pT+\langle\beta\nabla u\cdot\bn-Q_b(\beta\nabla Q_0u\cdot\bn),Q_bv_0-v_b\rangle_\pT\\
&=&\sum_T \langle\beta\nabla u\cdot\bn-Q_b(\beta\nabla u\cdot\bn),v_0-Q_bv_0\rangle_\pT+\langle\beta\nabla u\cdot\bn-\beta\nabla Q_0u\cdot\bn,Q_bv_0-v_b\rangle_\pT\\
&\le& \sum_T(h^{1/2}\|\beta\nabla u\cdot\bn-Q_b(\beta\nabla u\cdot\bn)\|_\pT)(h^{-1/2}\|v_0-Q_bv_0\|_\pT)\\
&&+(h^{1/2}\|\beta\nabla u\cdot\bn-\beta\nabla Q_0u\cdot\bn\|_\pT)(h^{-1/2}\|Q_bv_0-v_b\|_\pT)
\\
&\le& Ch \|u\|_{\tilde H^2(\Omega)}\sum_T\|\nabla v_0\|_T + \left(\sum_{T\in \mathcal{T}_h^I}\Big(h^3\|u\|_{\tilde{H}^3(\Omega)}^2+h^2\|u\|_{\tilde{H}^2(T)}^2\Big)+\sum_{T\in \mathcal{T}_h^R}h^2\|u\|_{{H}^2(T)}^2\right)^{1/2}\3bar v\3bar\\
&\le& h\|u\|_{\tilde{H}^3(\Omega)}\3bar v\3bar.
\end{eqnarray*}
In the last step, we use the assumption $(H_4)$ that the number of interface elements is of order $O(h^{-1})$. Combining the error bounds for $I$ and $II$, we obtain \eqref{eq: error luv}.
\end{proof}

Now we are ready to prove our main result. 
\begin{theorem}
{Let $Q_hu$ and $u_h$ be solutions to (\ref{eq:luv}) and (\ref{eq:scheme}), respectively. Then the following error estimate holds
\begin{eqnarray}\label{error}
\3bar Q_h u-u_h\3bar\le Ch\|u\|_{\tilde H^{3}{(\Omega)}},
\end{eqnarray}
where the hidden constant is independent of the interface location.
}
\end{theorem}
\begin{proof}
Taking $v=Q_hu-u_h$ in error equation (\ref{eq:error-eq}), and then from the above estimates, (\ref{eq:stab-estimate}), and combining with the coercivity of $A(\cdot,\cdot)$, we have 
\begin{eqnarray*}
c\3bar Q_h u-u_h\3bar^2 
&\le& A(Q_hu-u_h,Q_h u-u_h)\\
&=&L_u(Q_h u-u_h)+S(Q_hu,Q_h u-u_h)\\
&\le& Ch\|u\|_{\tilde{H}^3(\Omega)} \3bar Q_h u-u_h\3bar+ Ch\|u\|_{\tilde{H}^2(\Omega)}\3bar Q_h u-u_h\3bar\\
&\le& Ch\|u\|_{\tilde{H}^3(\Omega)} \3bar Q_h u-u_h\3bar.
\end{eqnarray*}
\end{proof}

\begin{remark}
In the error estimates \eqref{eq: error luv} and \eqref{error}, we need to the assumption that the regularity of the solution is piecewise $H^3$, which is usually higher than the usual piecewise $H^2$ assumption for numerical methods based on linear polynomials. However, this is only necessary for theoretical error analysis. In computation, piecewise $H^2$ assumption is sufficient to gain optimal convergence rate. 
\end{remark}
%
%

\section{Numerical Examples}

In this section, we report some numerical examples to validate our theoretical results. Furthermore, we will report the convergence test of the numerical solution in other norms. Let the exact solution be $u = (u_0,u_b)$, and the immersed weak Galerkin solution be $u_h = (u_{0h}, u_{bh})$. For simplicity, we also define the errors 
$$e_0 = u_0 - u_{0h},~~~
e_b = u_b - u_{bh}.$$
We will test the $L^\infty$, $L^2$, and semi-$H^1$ norms of $e_0$, and $L^\infty$ norm of $e_b$ in the following examples.
\begin{equation}
\|e_0\|_{L^2} =  \bigg(\sum_{T\in\mathcal{T}_h}\|u-u_0\|_T^2\bigg)^{1/2},~~
|e_0|_{H^1} = \bigg(\sum_{T\in\mathcal{T}_h}\|\nabla(u-u_0)\|_T^2\bigg)^{1/2},~~
\end{equation}
\begin{equation}
\|e_0\|_{L^\infty} =  \max_{x\in\mathcal{N}_h}\|u_0(x)-u_{0h}(x)\|,~~
\|e_b\|_{L^\infty} =  \max_{x\in\mathcal{M}_h}\|u_b(x)-u_{bh}(x)\|,~~
\end{equation}
where $\mathcal{N}_h$ and $\mathcal{M}_h$ denote the set of nodes of the mesh, and the set of midpoints of all edges of the mesh, respectively. We note that the semi-$H^1$ norm of $e_0$ is equivalent to the energy norm that we considered in the analysis. Thus, one can expect the errors measured by them give the same convergence rates. In all the numerical experiments, we take $\rho=10$. 

\subsection{Example 1} \label{sec:ex1}
We first consider a bench mark example for the elliptic interface problem which has been tested in many articles \cite{2015LinLinZhang, 2015LinYangZhang1}. Let $\Omega=(-1,1)\times(-1,1)$, which is divided into two subdomains $\Omega^{-}$ and $\Omega^{+}$ by a circular interface $\Gamma$ centered at origin with radius $r_0=\pi/5$ such that $\Omega^{-}=\{(x,y):x^2+y^2< r_0^2\}$ and $\Omega^{+}=\{(x,y):x^2+y^2> r_0^2\}$. Functions $f$ and $g$ are computed such that the analytical solution is described as follows:
\begin{eqnarray}
u(x,y)=\begin{cases}
\frac{1}{\beta^{-}}r^{\alpha}, & (x,y)\in\Omega^{-}\\
\frac{1}{\beta^{+}}r^{\alpha}+\bigg(\frac{1}{\beta^{-}}-\frac{1}{\beta^{+}}\bigg)r_0^{\alpha} & (x,y)\in\Omega^{+},
\end{cases}
\end{eqnarray}
where $r=\sqrt{x^2+y^2}$ and $\alpha=5$. We use the uniform Cartesian triangular meshes which is obtained by first partitioning the domain into $N\times N$ congruent rectangles and then connecting the top-left and bottom-right diagonal in every rectangle. We only report numerical performance for large coefficient contrasts $(\beta^-,\beta^+) = (1, 1000)$ and $(\beta^-,\beta^+) = (1000, 1)$. We also have testes some small coefficient jumps, and the numerical result is similar, hence we omit them in the paper. The numerical errors and convergence rates for these two cases are reported in Table \ref{table: error_circle_1_1000} and Table \ref{table: error_circle_1000_1}, respectively. The numerical solutions on the $128\times 128$ mesh are plotted in Figure \ref{fig: Ex1 sol}. From these Tables, we can observe clearly that the error $e_0$ in semi-$H^1$ norm converge optimally which confirms our theoretical analysis. Moreover, $e_0$ in $L^2$ and $L^\infty$ norms also converge in second-order, which is considered as optimal rate. The $e_b$ in $L^\infty$ norm seems to converge in first order, which is expected as we use the piecewise constant approximation for $u_b$.

\begin{figure}[!tb]
\begin{center}
\includegraphics[width=0.45\textwidth]{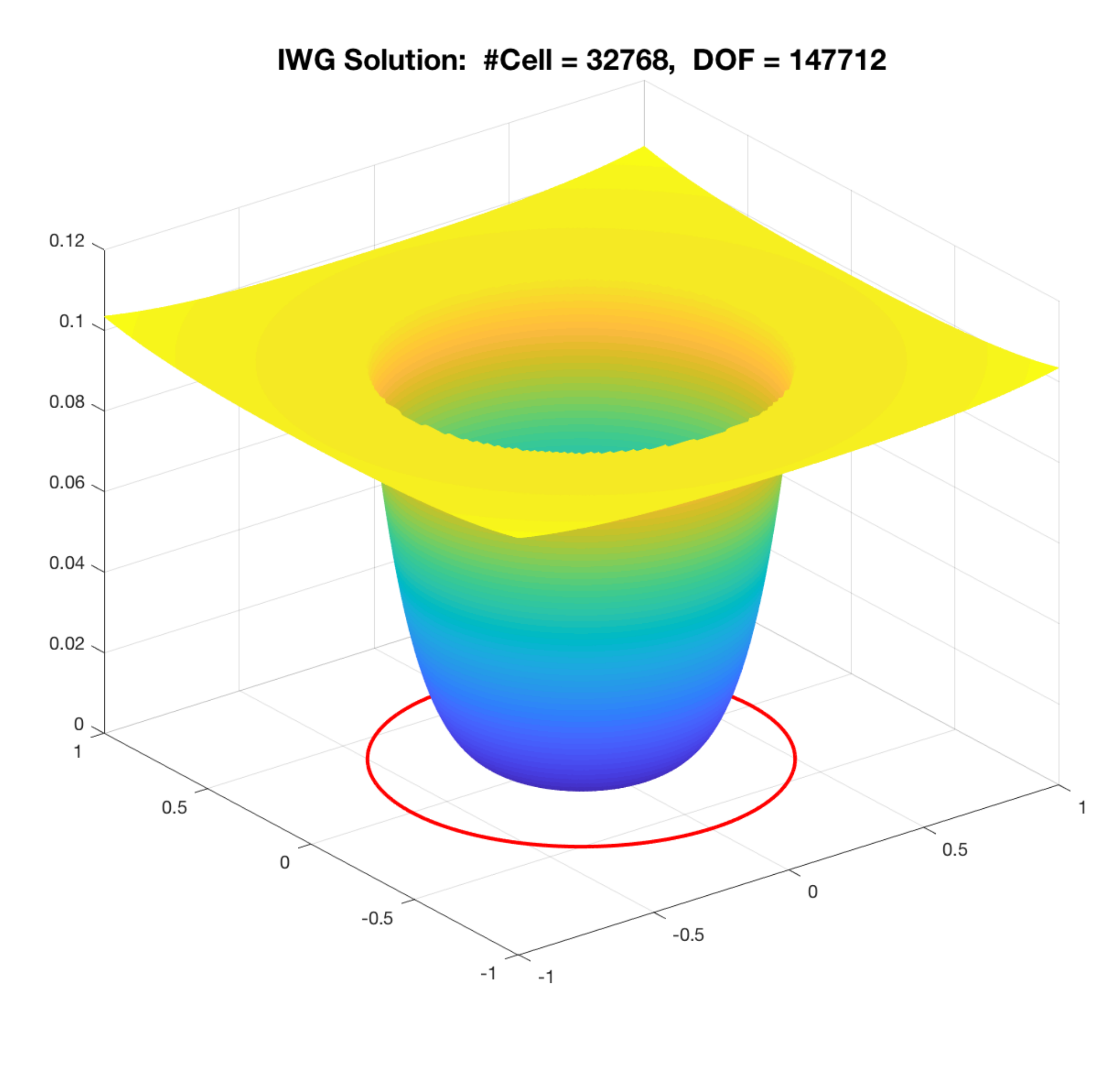}
\includegraphics[width=0.45\textwidth]{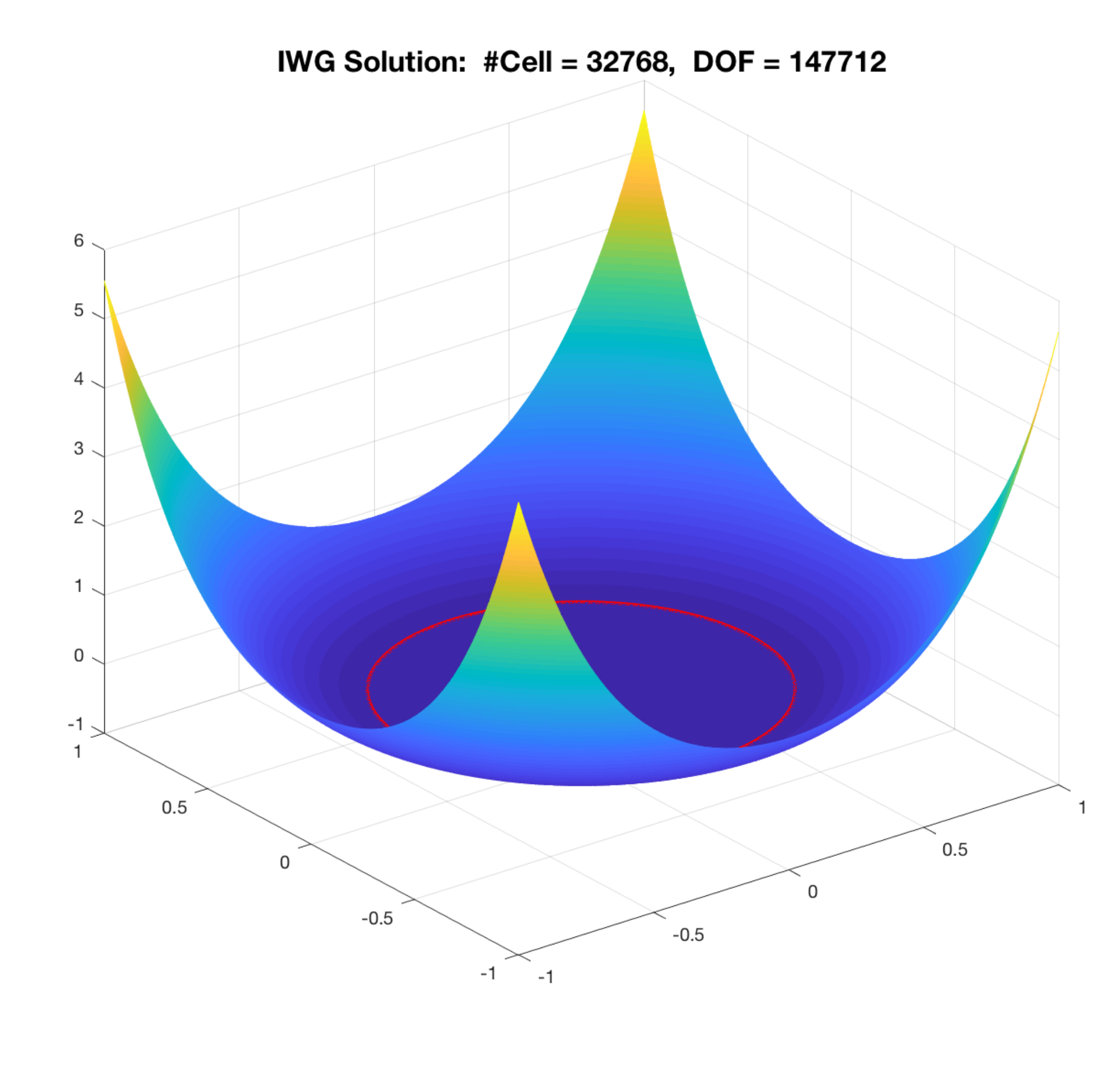}
\end{center}
\caption{Immersed Weak Galerkin solutions for Example 5.1 with $(\beta^-,\beta^+) = (1,1000)$, and $(\beta^-,\beta^+) = (1000,1)$}
\label{fig: Ex1 sol}
\end{figure}

\begin{table}
\begin{center}
\begin{tabular}{|c|c|cc|cc|cc|cc|} \hline
$N$          & DOF       &  $\|e_0\|_{\infty}$  &Order & $\|e_b\|_{\infty}$ &Order &$\|e_0\|_{L^2}$&Order    & $\|e_0\|_{H^1}$&Order \\ 
\hline
$16$     & 2.34E+3 &  1.74E-2 &              & 1.13E-2 &             &2.99E-3 &               &1.04E-1 & \\
$32$     & 9.28E+3 &  5.21E-3 & 1.74      & 6.57E-2 & 0.78     &7.81E-4 & 1.93       &4.89E-2 & 1.08 \\
$64$     & 3.70E+4 &  1.54E-3 & 1.75      & 4.06E-3 & 0.69     &1.99E-4 & 1.97       &2.44E-2 & 1.00  \\
$128$   & 1.48E+5 &  4.34E-5 & 1.83      & 1.90E-3 & 1.09     &5.11E-5 & 1.96       &1.25E-2 & 0.97  \\
$256$   & 5.90E+5 &  1.13E-4 & 1.94      & 9.90E-4 & 0.94     &1.28E-5 & 2.00       &6.28E-3 & 0.99  \\ 
$512$   & 2.26E+6 &  3.16E-5 & 1.84      & 5.05E-3 & 0.97     &3.22E-6 & 1.99       &3.15E-3 & 1.00  \\ 
$1024$ & 9.44E+6 &  7.89E-6 & 2.00      & 2.63E-4 & 0.94     &8.09E-7 & 1.99       &1.57E-3 & 1.01  \\
\hline
\end{tabular}
\caption{Errors of Immersed WG methods Circle Interface for $\beta^-=1$, $\beta^+=1000$}
\label{table: error_circle_1_1000}
\end{center}
\end{table}

\begin{table}
\begin{center}
\begin{tabular}{|c|c|cc|cc|cc|cc|} \hline
$N$          & DOF       &  $\|e_0\|_{\infty}$  &Order & $\|e_b\|_{\infty}$ &Order &$\|e_0\|_{L^2}$&Order    & $\|e_0\|_{H^1}$&Order \\ 
\hline
$16$     & 2.34E+3 &  1.81E-1 &              & 2.24E-2 &             &3.13E-2 &               &1.15E-0 & \\
$32$     & 9.28E+3 &  4.83E-2 & 1.91      & 8.81E-3 & 1.35     &7.89E-3 & 1.99       &5.76E-1 & 1.00 \\
$64$     & 3.70E+4 &  1.25E-3 & 1.95      & 4.62E-3 & 0.93     &1.98E-3 & 2.00       &2.88E-1 & 1.00  \\
$128$   & 1.48E+5 &  3.17E-3 & 1.98      & 2.03E-3 & 1.19     &4.94E-4 & 2.00       &1.44E-1 & 1.00  \\
$256$   & 5.90E+5 &  8.01E-4 & 1.99      & 1.02E-3 & 0.99     &1.23E-4 & 2.00       &7.20E-2 & 1.00  \\ 
$512$   & 2.26E+6 &  2.01E-4 & 1.99      & 5.13E-4 & 1.00     &3.09E-5 & 2.00       &3.60E-2 & 1.00  \\ 
$1024$ & 9.44E+6 &  5.04E-5 & 2.00      & 2.65E-4 & 0.96     &7.73E-6 & 2.00       &1.80E-2 & 1.00  \\
\hline
\end{tabular}
\caption{Errors of Immersed WG methods Circle Interface for $\beta^-=1000$, $\beta^+=1$}
\label{table: error_circle_1000_1}
\end{center}
\end{table}

%

\subsection{Example 2}
In this example, we test our numerical algorithm for a more complicated interface curve. We let $\Omega=(-1,1)\times(-1,1)$, and the interface is determined by the following level-set function: 
\begin{eqnarray}
    \Gamma(x,y)= (x^2 + y^2)^2(1 + 0.4\sin(6\arctan\big(\frac{y}{x}\big)))-0.3.
\end{eqnarray}
The subdomains are defined as $\Omega^+ = \{(x,y): \Gamma(x,y)>0\}$, and  $\Omega^- = \{(x,y): \Gamma(x,y)<0\}$. The exact solution is chosen as:
\begin{eqnarray}
u=\begin{cases}
\dfrac{1}{\beta^-}\Gamma(x,y) ,\ (x,y)\in\Omega^-\\
\dfrac{1}{\beta^+}\Gamma(x,y),\ (x,y)\in\Omega^+.
\end{cases}
\end{eqnarray} 
We test the high coefficient jump cases $(\beta^-,\beta^+) = (1, 1000)$ and $(\beta^-,\beta^+) = (1000, 1)$, and the error tables are reported in Table \ref{table: error_petal_1_1000}, and Table \ref{table: error_petal_1000_1}, respectively. The numerical solutions on the $128\times 128$ mesh are plotted in Figure \ref{fig: Ex2 sol}. From these data, we can observe again that the error in $H^1$- and $L^2$-norms converge in first and second order, respectively. And the infinity norm for $u_0$ and $u_b$ are close to second order and first order.  

\begin{figure}[htb]
\begin{center}
\includegraphics[width=0.45\textwidth]{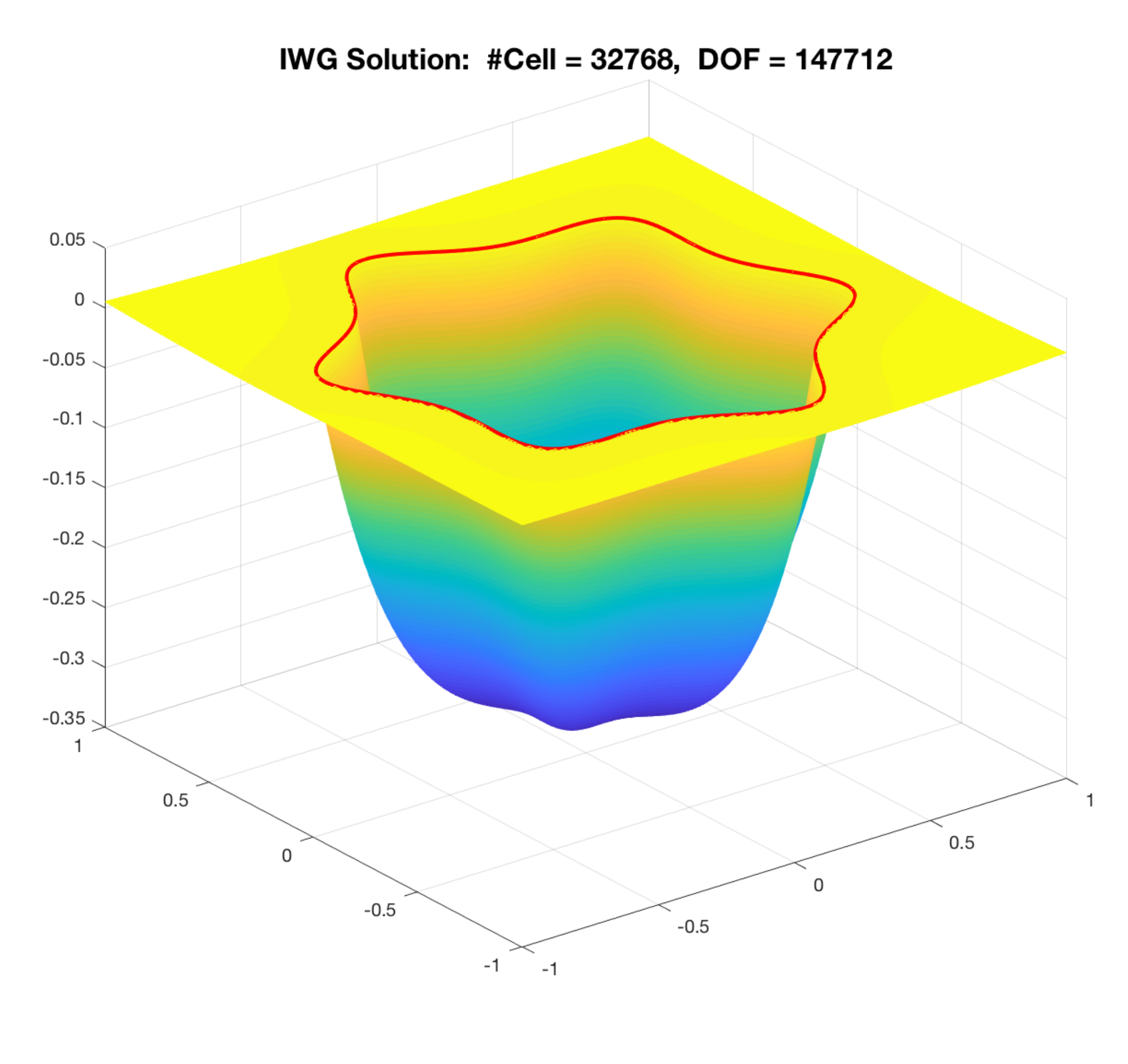}
\includegraphics[width=0.45\textwidth]{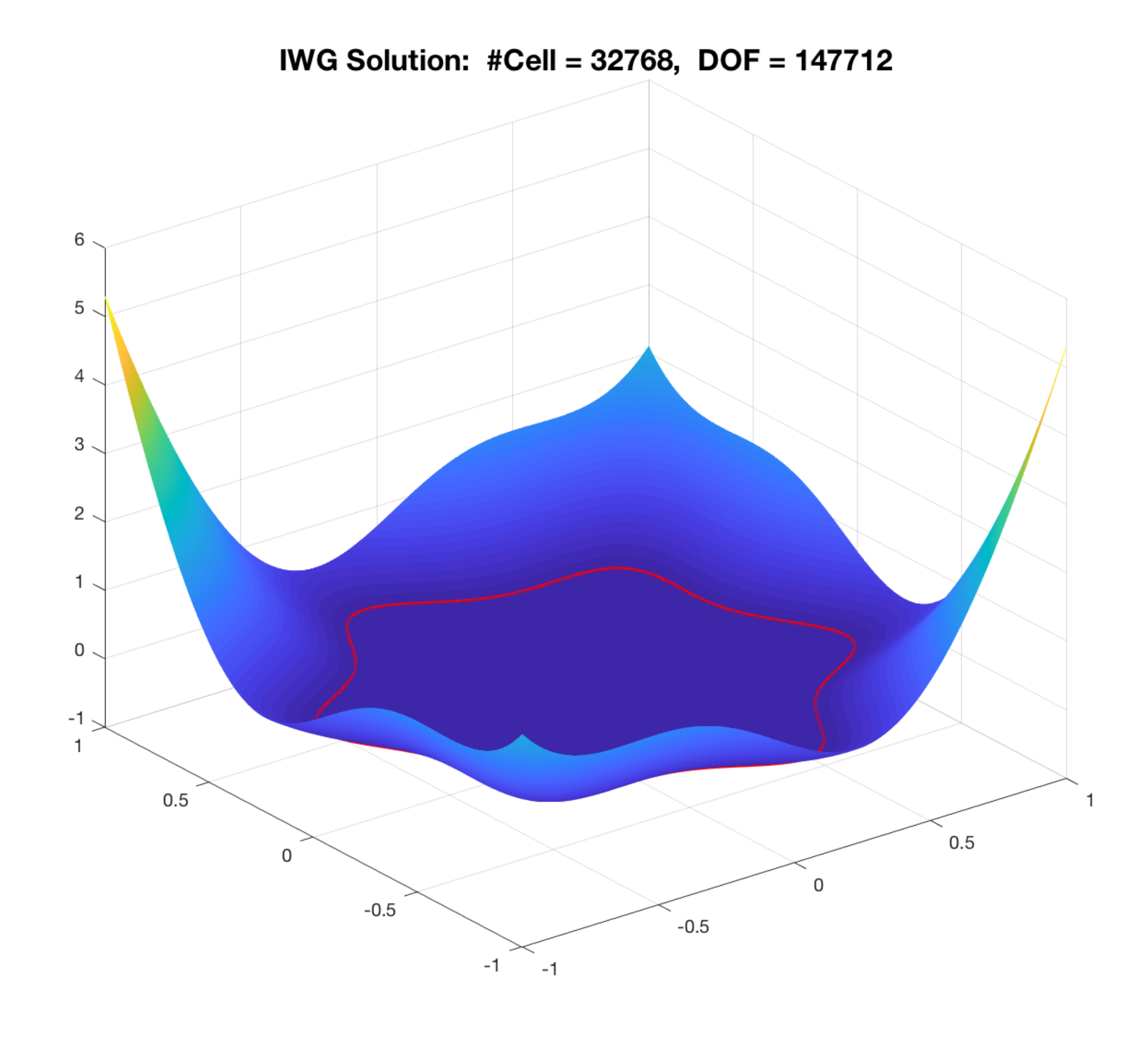}
\end{center}
\caption{Immersed Weak Galerkin solutions for Example 5.2 with $(\beta^-,\beta^+) = (1,1000)$, and $(\beta^-,\beta^+) = (1000,1)$}\label{fig: Ex2 sol}
\end{figure}

\begin{table}[htb]
\begin{center}
\begin{tabular}{|c|c|cc|cc|cc|cc|} \hline
$N$          & DOF       &  $\|e_0\|_{\infty}$  &Order & $\|e_b\|_{\infty}$ &Order &$\|e_0\|_{L^2}$&Order    & $\|e_0\|_{H^1}$&Order \\ 
\hline
$16$     & 2.34E+3 &  5.00E-2 &              & 1.79E-2 &             &8.46E-3 &               &2.89E-1 & \\
$32$     & 9.28E+3 &  1.53E-2 & 1.71      & 1.58E-2 & 0.18     &2.30E-3 & 1.88       &1.51E-1 & 0.94 \\
$64$     & 3.70E+4 &  4.36E-3 & 1.81      & 7.68E-3 & 1.04     &5.87E-4 & 1.97       &7.51E-2 & 1.01  \\
$128$   & 1.48E+5 &  1.38E-3 & 1.66      & 4.63E-3 & 0.73     &1.60E-4 & 1.88       &3.68E-2 & 1.03  \\
$256$   & 5.90E+5 &  4.36E-4 & 1.66      & 2.43E-3 & 0.93     &4.07E-5 & 1.97       &1.86E-2 & 0.99  \\ 
$512$   & 2.26E+6 &  1.36E-4 & 1.68      & 1.18E-3 & 1.04     &1.03E-5 & 1.99       &9.21E-3 & 1.01  \\ 
$1024$ & 9.44E+6 &  3.98E-5 & 1.77      & 6.00E-4 & 0.98     &2.60E-6 & 1.98       &4.58E-3 & 1.01  \\
\hline
\end{tabular}
\caption{Errors of Immersed WG methods Petal Interface for $\beta^-=1$, $\beta^+=1000$}
\label{table: error_petal_1_1000}
\end{center}
\end{table}

\begin{table}[htb]
\begin{center}
\begin{tabular}{|c|c|cc|cc|cc|cc|} \hline
$N$          & DOF       &  $\|e_0\|_{\infty}$  &Order & $\|e_b\|_{\infty}$ &Order &$\|e_0\|_{L^2}$&Order    & $\|e_0\|_{H^1}$&Order \\ 
\hline
$16$     & 2.34E+3 &  1.40E-1 &              & 3.85E-2 &             &3.01E-2 &               &9.50E-1 & \\
$32$     & 9.28E+3 &  3.81E-2 & 1.88      & 1.68E-2 & 1.20     &7.86E-3 & 1.93       &4.84E-1 & 0.97 \\
$64$     & 3.70E+4 &  9.86E-3 & 1.95      & 8.03E-3 & 1.06     &2.02E-3 & 1.96       &2.43E-1 & 0.99  \\
$128$   & 1.48E+5 &  2.51E-3 & 1.98      & 4.64E-3 & 0.79     &5.06E-4 & 1.99       &1.21E-2 & 1.00  \\
$256$   & 5.90E+5 &  7.41E-4 & 1.76      & 2.43E-3 & 0.93     &1.26E-4 & 2.01       &6.07E-2 & 1.00  \\ 
$512$   & 2.26E+6 &  1.58E-4 & 2.23      & 1.19E-3 & 1.03     &3.15E-5 & 2.00       &3.03E-2 & 1.00  \\ 
$1024$ & 9.44E+6 &  3.97E-5 & 2.00      & 6.01E-4 & 0.98     &7.87E-6 & 2.00       &1.52E-2 & 1.00  \\
\hline
\end{tabular}
\caption{Errors of Immersed WG methods Petal Interface for $\beta^-=1$, $\beta^+=1000$}
\label{table: error_petal_1000_1}
\end{center}
\end{table}

\subsection{Example 3} 
In this example, we consider the case when the interface has a sharp corner. This example has been used in \cite{2010KwakWeeChang}. Let $\Omega=(-1,1)\times(-1,1)$, and the interface is defined by the level-set function:
\begin{eqnarray}
\Gamma(x,y)=-y^2+((x-1)\tan(\theta))^2x.
\end{eqnarray}
The subdomains are defined as $\Omega^+ = \{(x,y): \Gamma(x,y)>0\}$, and  $\Omega^- = \{(x,y): \Gamma(x,y)<0\}$. The exact solution is chosen as:
\begin{eqnarray}\label{eq:ex3-sol}
u=\begin{cases}
\dfrac{1}{\beta^-}\Gamma(x,y) ,\ (x,y)\in\Omega^-,\\
\dfrac{1}{\beta^+}\Gamma(x,y),\ (x,y)\in\Omega^+.
\end{cases}
\end{eqnarray} 
The right hand function $f$ is chosen accordingly to fit the exact solution $u(x, y)$ in \eqref{eq:ex3-sol}. We note that on the point $(1,0)$, the interface curve has a sharp corner. We slightly adjust our uniform mesh such that an odd number of partition in each direction. By doing this, the singular point will be located in one of the mesh point. The performance of our proposed numerical scheme is reported in Table \ref{table: error_singular_1_1000}-\ref{table: error_singular_1000_1}. Similar conclusions as previous ones can be made for such convergence tests. Furthermore, the numerical solutions are plotted in Figure \ref{fig: Ex3 sol} for varying values in $\beta$.

\begin{figure}[htb]
\begin{center}
\includegraphics[width=0.45\textwidth]{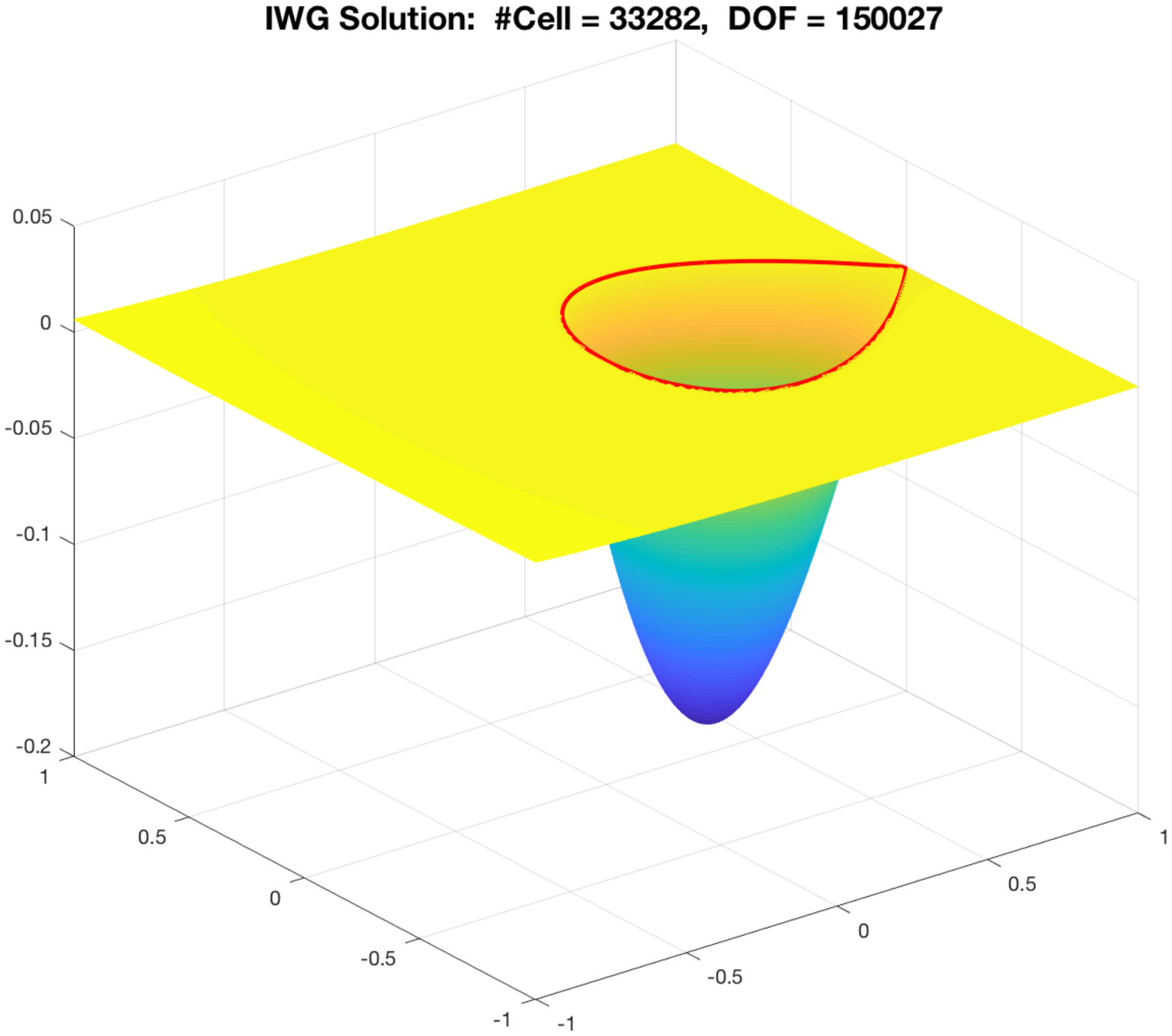}
\includegraphics[width=0.45\textwidth]{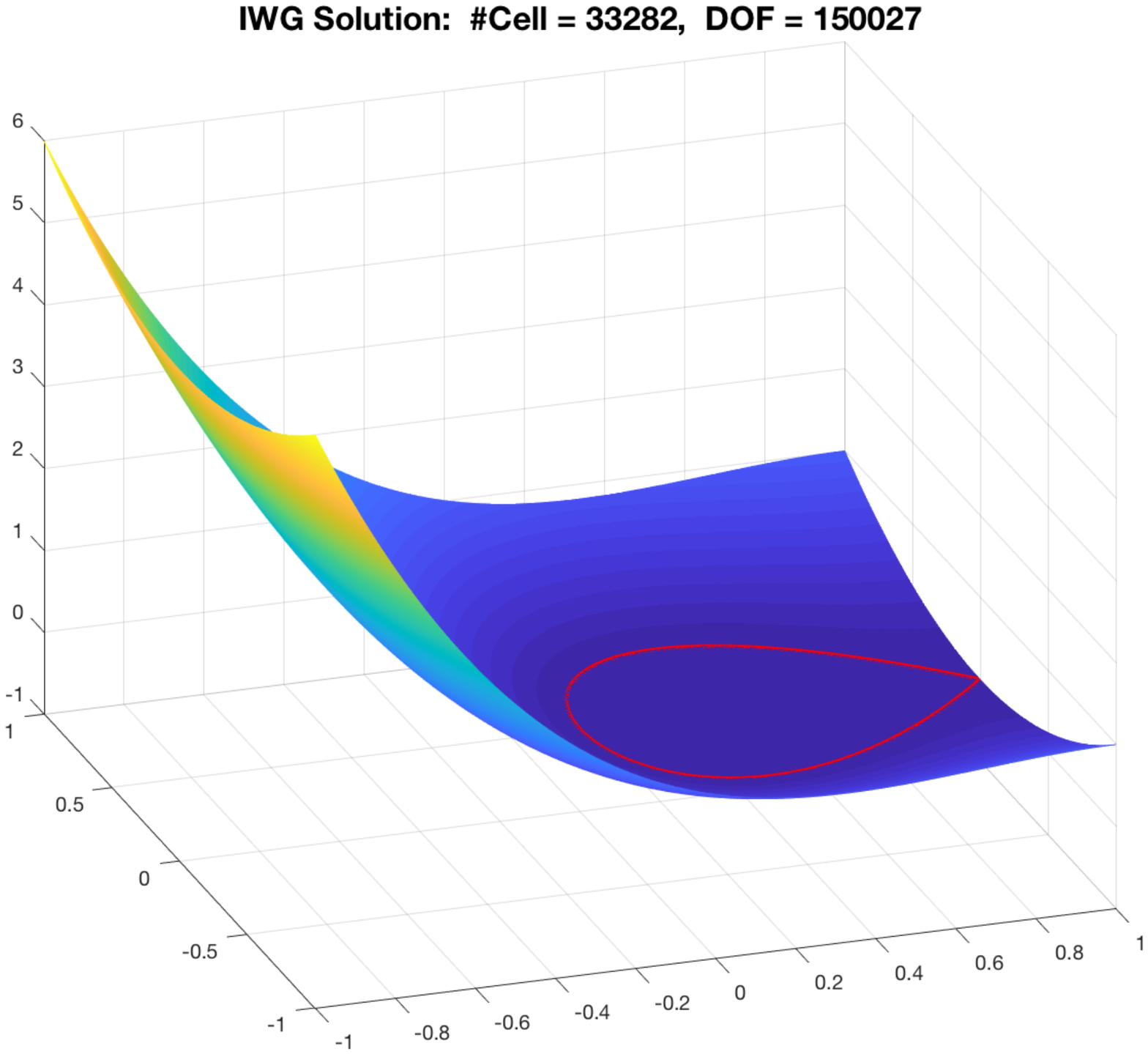}
\end{center}
\caption{Immersed Weak Galerkin solutions for Example 5.3 with $(\beta^-,\beta^+) = (1,1000)$, and $(\beta^-,\beta^+) = (1000,1)$}\label{fig: Ex3 sol}
\end{figure}

\begin{table}[htb]
\begin{center}
\begin{tabular}{|c|c|cc|cc|cc|cc|} \hline
$N$          & DOF       &  $\|e_0\|_{\infty}$  &Order & $\|e_b\|_{\infty}$ &Order &$\|e_0\|_{L^2}$&Order    & $\|e_0\|_{H^1}$&Order \\ \hline
$17$     & 2.64E+3 &  6.55E-2 &              & 1.72E-2 &             &1.65E-2 &               &4.88E-1 & \\
$33$     & 9.87E+3 &  1.77E-2 & 1.89      & 7.47E-3 & 1.20     &4.41E-3 & 1.90       &2.52E-1 & 0.95 \\
$65$     & 3.82E+4 &  4.60E-3 & 1.94      & 5.15E-3 & 0.54     &1.11E-3 & 1.98       &1.28E-1 & 0.98  \\
$129$   & 1.50E+5 &  1.17E-3 & 1.97      & 2.44E-3 & 1.08     &2.75E-4 & 2.02       &6.45E-2 & 0.99  \\
$257$   & 5.95E+5 &  2.96E-4 & 1.99      & 1.27E-3 & 0.95     &6.88E-5 & 2.00       &3.24E-3 & 0.99  \\ 
$513$   & 2.37E+6 &  7.44E-5 & 1.99      & 6.96E-4 & 0.86     &1.71E-5 & 2.00       &1.62E-3 & 1.00  \\ 
$1025$ & 9.46E+6 &  1.87E-5 & 2.00      & 3.56E-4 & 0.97     &4.28E-6 & 2.00       &8.11E-3 & 1.00  \\ 
\hline
\end{tabular}
\caption{Errors of Immersed WG methods Singular Interface for $\beta^-=1$, $\beta^+=1000$}
\label{table: error_singular_1_1000}
\end{center}
\end{table}

\begin{table}[htb]
\begin{center}
\begin{tabular}{|c|c|cc|cc|cc|cc|} \hline
$N$          & DOF       &  $\|e_0\|_{\infty}$  &Order & $\|e_b\|_{\infty}$ &Order &$\|e_0\|_{L^2}$&Order    & $\|e_0\|_{H^1}$&Order \\ 
\hline
$17$     & 2.64E+3 &  1.48E-2 &              & 1.49E-2 &             &2.41E-3 &               &8.58E-2 & \\
$33$     & 9.87E+3 &  5.79E-3 & 1.36      & 8.62E-3 & 0.79     &6.73E-4 & 1.84       &4.42E-2 & 0.96 \\
$65$     & 3.82E+4 &  2.17E-3 & 1.42      & 5.30E-3 & 0.70     &1.69E-4 & 2.00       &2.28E-2 & 0.95  \\
$129$   & 1.50E+5 &  5.20E-4 & 2.06      & 2.46E-3 & 1.11     &3.80E-5 & 2.15       &1.09E-2 & 1.06  \\
$257$   & 5.95E+5 &  1.21E-4 & 2.10      & 1.27E-3 & 0.95     &9.25E-6 & 2.04       &5.48E-3 & 1.00  \\ 
$513$   & 2.37E+6 &  3.14E-5 & 1.95      & 6.95E-4 & 0.87     &2.25E-6 & 2.04       &2.71E-3 & 1.01  \\ 
$1025$ & 9.46E+6 &  8.06E-6 & 1.96      & 3.56E-4 & 0.97     &5.58E-7 & 2.01       &1.36E-3 & 1.00  \\ 
\hline
\end{tabular}
\caption{Errors of Immersed WG methods Singular Interface for $\beta^-=1000$, $\beta^+=1$}
\label{table: error_singular_1000_1}
\end{center}
\end{table}


\end{document}